\documentclass[12pt,a4paper]{amsart}
\usepackage{amssymb,enumerate}
\usepackage{tabularx}
\usepackage{tikz-cd}
\usepackage{geometry}
 \usepackage{color}

 \newcommand{\cl}{\operatorname{cl}}
\newcommand{\CC}{{\mathbb{C}}}

\newcommand{\GG}{{\mathbb{G}}}

\newcommand{\Q}{\mathbb{Q}}

\newcommand{\CF}{\operatorname{CF}}

\newcommand{\Z}{{\mathbb{Z}}}
\newcommand{\ZZ}{{\mathbb{Z}}}
\newcommand{\cf}{\operatorname{cf}}
\newcommand{\ba} {\mathbf a}

\newcommand{\bk} {\mathbf k}

\newcommand{\bu} {\mathbf u}

\newcommand{\bz} {\mathbf z}
\newcommand{\Zp}{\mathbb{Z}_{(p)}}

\newcommand{\cF} {\mathcal F}
\newcommand{\cG} {\mathcal G}

\newcommand{\cN} {\mathcal N}

\newcommand{\cX} {\mathcal X}
\newcommand{\cY} {\mathcal Y}

\newcommand{\Fz}{\cF^{z}}

\newcommand{\Hom}{{\operatorname{Hom}}}

\newcommand{\Ch}{{\operatorname{Ch}}}

\newcommand{\Ind}{{\operatorname{Ind}}}

\newcommand{\Irr}{{\operatorname{Irr}}}

\newcommand{\PSp}{{\operatorname{PSp}}}

\newcommand{\SL}{{\operatorname{SL}}}

\newcommand{\Syl}{{\operatorname{Syl}}}

\newcommand{\Vol}{{\operatorname{Vol}}}

\makeatletter
\newcommand{\colim@}[2]{%
  \vtop{\m@th\ialign{##\cr
    \hfil$#1\operator@font colim$\hfil\cr
    \noalign{\nointerlineskip\kern1.5\ex@}#2\cr
    \noalign{\nointerlineskip\kern-\ex@}\cr}}%
}
\newcommand{\colim}{%
  \mathop{\mathpalette\colim@{\rightarrowfill@\textstyle}}\nmlimits@
}
\makeatother

\newcommand{\Res}{{\operatorname{Res}}}
\newcommand{\reg}{{\operatorname{reg}}}
\newcommand{\Aut}{{\operatorname{Aut}}}

\newcommand{\rk}{{\operatorname{rk}}}

\newcommand{\tr}{{\operatorname{tr}}}

\newcommand{\GL}{\operatorname{GL}}

\newtheorem{thm}{Theorem}[section]

\newtheorem{lem}[thm]{Lemma}
\newtheorem{cor}[thm]{Corollary}
\newtheorem{prop}[thm]{Proposition}

\newcommand{\lb}{\langle}
\newcommand{\rb}{\rangle}

\theoremstyle{definition}
\newtheorem{rem}[thm]{Remark}

\newtheorem{exmp}[thm]{Example}

\numberwithin{equation}{section}

\begin{document}

%%%%%%%%%%%%%%%%%%%%%%%%%%%%%%%%%%%%%%%%%%%%%%%%%%%%%%%%%%%%%
\title{On character tables for fusion systems}
%%%%%%%%%%%%%%%%%%%%%%%%%%%%%%%%%%%%%%%%%%%%%%%%%%%%%%%%%%%%

\author{Thomas Lawrence}
\address{Department of Mathematics, Loughborough University, LE11 3TT,
  U.K.}
\email{t.lawrence@lboro.ac.uk}

\author{Jason Semeraro}
\address{Department of Mathematics, Loughborough University, LE11 3TT,
  U.K.}
\email{j.p.semeraro@lboro.ac.uk}

%\thanks{}

\begin{abstract}
A character table $X$ for a saturated fusion system $\cF$ on a finite $p$-group $S$ is the square matrix of values associated to a basis of the lattice of virtual $\cF$-stable ordinary characters of $S$. We prove a conjecture of the second author which equates the determinant of $X\overline{X}^T$ (the square of the volume of this lattice) with the product of the orders of $S$-centralisers of fully $\cF$-centralised $\cF$-class representatives. This statement is exactly column orthogonality for the character table of $S$ when $\cF=\cF_S(S)$. We first give an elementary proof when $\cF=\cF_S(G)$ is realised by a finite group $G$ with Sylow $p$-subgroup $S$, and then treat the general case by exploiting the characteristic idempotent $\omega_\cF$ of $\cF$ in the $p$-local double Burnside ring. The two ingredients are that $\omega_\cF$ acts on class functions as the operator of averaging over $\cF$-classes, which forces the lattice of $\cF$-stable characters to be self-dual at $p$, and that this averaging operator preserves $\ZZ_{(p)}$-valued class functions, which forces $|x^\cF|_p=[S:C_S(x)]$ for every fully $\cF$-centralised $x \in S$. The latter equality was conjectured in an earlier version of this paper.
\end{abstract}



\date{\today}

\maketitle

\pagestyle{myheadings}
%\markboth{for personal use only}{preliminary}

%%%%%%%%%%%%%%%%%%%%%%%%%%%%

\section{Introduction}
\label{s:intro}
Let $p$ be a prime, $S$ a finite $p$-group and $\cF$ a fusion system on $S$. For $x\in S$ write
\[
x^{\cF}=\{\varphi(x): \varphi\in\Hom_{\cF}(\langle x\rangle,S)\}
\]
for the $\cF$-conjugacy class of $x$ and $x^S$ for its $S$-conjugacy class. An element $x$ is \emph{fully $\cF$-centralised} if $|C_S(x)|\ge |C_S(y)|$ for all $y\in x^{\cF}$. Let $\Fz\subseteq S$ be a set of fully $\cF$-centralised representatives for the $\cF$-classes of elements of $S$, and put $\mathbf{k}(\cF)=|\Fz|$, the number of $\cF$-classes.

Let $\Ch(S)=\ZZ[\Irr(S)]$ be the ring of virtual ordinary characters of $S$, viewed as a lattice in the space $\cf(S)$ of complex class functions on $S$. A virtual character is \emph{$\cF$-stable} if it is constant on every $\cF$-class, and $\Ch(S)^{\cF}\subseteq\Ch(S)$ denotes the sublattice of $\cF$-stable virtual characters. It has rank $\mathbf{k}(\cF)$ (Corollary~\ref{cor:rank} below, or \cite[Lem. 2.1]{BC20}). Given a $\ZZ$-basis $B$ of $\Ch(S)^\cF$, the \emph{character table of $\cF$ with respect to $B$} is
\[
X=X_B(\cF)\in M_{\mathbf{k}(\cF)}(\CC),\qquad X_{\chi,s}=\chi(s)\quad(\chi\in B,\ s\in\Fz),
\]
and the square of the volume of the lattice $\Ch(S)^\cF$ is $|\det X|^2=\det\big(X\overline{X}^{\,T}\big)$. This quantity is independent of the choice of $B$ (Lemma~\ref{lem:basisindep}). The following result was conjectured by the second author (see \cite[Conj. 1.1]{L24}); it is the main theorem of this paper.

\begin{thm}\label{t:main}
Let $\cF$ be a saturated fusion system on a finite $p$-group $S$ and let $B$ be a $\ZZ$-basis of $\Ch(S)^\cF$. Then
\[
\det\big(X_B(\cF)\overline{X_B(\cF)}^{\,T}\big)\;=\;\prod_{s\in\Fz}|C_S(s)|.
\]
\end{thm}

We make some remarks concerning various special cases. Firstly, when $\cF=\cF_S(S)$ the $\cF$-classes are the $S$-classes, $B=\Irr(S)$, and the statement is exactly the column orthogonality relation for the character table of $S$. Secondly, recall that $\cF$ is \textit{transitive} on $S$ if $\cF^z=\{1,z\}$ for some $1 \neq z \in Z(S)$. For such a fusion system the set $B=\{1_S, \reg_S-1_S\}$ is a basis for $\Ch(S)^\cF$, where $\reg_S$ denotes the regular character of $S$ (see \cite[Lemma 2.16]{CC23}). This means that $X_B(\cF)=\left(\begin{smallmatrix} 1 & 1 \\ |S|-1 & -1\end{smallmatrix}\right)$ and Theorem \ref{t:main} holds trivially. In particular this gives a straightforward proof for the three Ruiz--Viruel exotic fusion systems on $7$-groups of order $7^3$ appearing in \cite{RV04}, all of which are transitive.

Thirdly, when $\cF=\cF_S(G)$ is realised by a finite group $G$ with $S\in\Syl_p(G)$, Theorem \ref{t:main} is a result of the first author \cite{L24} and, earlier and in the language of $\pi$-blocks, of Olsson \cite{O82}. In Section \ref{s:groupcase} we give a short proof which uses only Brauer's characterisation of characters and column orthogonality for $G$:

\begin{thm}\label{t:groupcase}
Let $G$ be a finite group with Sylow $p$-subgroup $S$ and let $\cF=\cF_S(G)$. Then Theorem \ref{t:main} holds for $\cF$.
\end{thm}

That proof uses in an essential way that $S$ is a Sylow $p$-subgroup of $G$. For an arbitrary $p$-subgroup $S\le G$ the computation leading to \eqref{e:2} still shows that $\det(X\overline{X^T})$ is an integer, and an argument with $\ell$-Brauer characters due to the first author shows that it is in fact a $p$-power; since every fusion system is realised by some finite group $G$ containing $S$ by a theorem of Park \cite{P16}, this holds for every fusion system, saturated or not (see Remark \ref{r:integral}). The saturation hypothesis in Theorem \ref{t:main} is, however, genuinely needed: see Example \ref{ex:c8}.

The general case is treated by the strategy proposed in \cite[Section 5]{LS25}, the earlier version of this paper, which is to exploit the characteristic idempotent $\omega_\cF\in A(S,S)_{(p)}$ of $\cF$ in the $p$-local double Burnside ring. Write $\pi$ for the operator induced by $\omega_\cF$ on the space $\cf(S)$ of complex class functions. Everything rests on one formula (Proposition \ref{prop:average}):
\[
(\pi f)(g)=\frac{1}{|g^{\cF}|}\sum_{h\in g^{\cF}}f(h)\qquad\big(f\in\cf(S),\ g\in S\big),
\]
that is, $\pi$ is the operator of averaging a class function over $\cF$-classes. This is a direct translation of Reeh's computation of the permutation character of $\omega_\cF$ \cite[Thm. B]{R16}. From it we extract two statements which at first sight look unrelated.

\emph{(i) A lattice-theoretic constraint (Section \ref{sec:selfdual}).} The averaging formula shows that $\pi$ is self-adjoint for the usual inner product on $\cf(S)$, and that it maps $\Ch(S)\otimes\Zp$ onto $\Ch(S)^\cF\otimes\Zp$. An orthogonal projection maps a unimodular lattice onto the dual of the sublattice it meets; since here image and intersection coincide, $\Ch(S)^{\cF}\otimes\Zp$ is self-dual. Equivalently, the Gram matrix $\Gamma_B=\big(\langle \psi,\phi\rangle\big)_{\psi,\phi\in B}$ has determinant prime to $p$. This replaces \cite[Qu. 5.2]{LS25}: no ``orthogonalising matrix'' is needed, because the idempotent supplies a $p$-adically unimodular Gram matrix directly.

\emph{(ii) A numerical constraint (Section \ref{sec:classes}).} Applying $\pi$ to the indicator function of an $S$-class gives $\pi(1_{x^S})=\frac{|x^S|}{|x^{\cF}|}1_{x^\cF}$. Since $\omega_\cF$ lies in the $p$-local double Burnside ring, $\pi$ preserves $\Zp$-valued class functions, so $|x^{\cF}|_p$ divides $|x^S|$. An elementary counting argument (Lemma \ref{lem:lower}) gives the opposite divisibility when $x$ is fully $\cF$-centralised, whence
\[
|x^{\cF}|_p=[S:C_S(x)] \qquad\text{for } x \text{ fully } \cF\text{-centralised},
\]
which is exactly \cite[Conj. 5.5]{LS25}.

These combine through the elementary identity
\[
\det\Gamma_B=\det\big(X\overline X^{\,T}\big)\prod_{s\in\Fz}\frac{|s^\cF|}{|S|}
\]
of Lemma \ref{lem:identity}, which pins down the $p$-part of $\det\big(X\overline X^{\,T}\big)$; knowing in addition that this determinant is a $p$-power then finishes the proof. In particular Theorem \ref{t:main} is \emph{equivalent} to \cite[Conj. 5.5]{LS25}; see Corollary \ref{cor:equiv}.

Apart from standard character theory and the elementary theory of lattices, the argument uses exactly the following three results from the literature, none of which is reproved here.

\begin{thm}[Bouc]\label{inp:bouc}
Every transitive bifree $(S,S)$-biset is isomorphic to $S\times_{(P,\varphi)}S$ for some $P\le S$ and injective homomorphism $\varphi\colon P\to S$.
\end{thm}

\begin{proof}
See \cite[Section 2.3]{B10} or \cite[Section 3]{R06}.
\end{proof}

\begin{thm}[Ragnarsson]\label{inp:rag}
A saturated fusion system $\cF$ on $S$ has a characteristic idempotent $\omega_\cF\in A(S,S)_{(p)}$, unique subject to being $\cF$-generated, $\cF$-stable and of augmentation prime to $p$.
\end{thm}

\begin{proof}
See \cite{R06} and \cite[Thm. A, Section 5]{RS13}.
\end{proof}

\begin{thm}[Reeh]\label{inp:reeh}
Let $\chi_{\omega_\cF}$ denote the $\ZZ_{(p)}$-localisation of the permutation character of $\omega_\cF$. Then for all $g,h\in S$,
\[
\chi_{\omega_\cF}(g,h)=\begin{cases} |S|\big/\big|\Hom_\cF(\lb g\rb,S)\big| & \text{if } h\in g^{\cF},\\ 0 & \text{otherwise.}\end{cases}
\]
\end{thm}

\begin{proof}
See \cite[Thm. B]{R16}.
\end{proof}

Theorems \ref{inp:rag} and \ref{inp:reeh} are the only places where saturation is used, and Theorem \ref{inp:bouc} is used only to make the action of a biset on class functions explicit.

\medskip \mbox{}\newline\textbf{Structure of the paper:} We recall background on $\cF$-stable characters in Section \ref{s:fstable} and prove Theorem \ref{t:groupcase} in Section \ref{s:groupcase}, following the approach taken in \cite{L24}. Sections \ref{s:bisets} and \ref{s:averaging} develop what we need about bisets and establish the averaging formula; Sections \ref{sec:selfdual} and \ref{sec:classes} prove ingredients (i) and (ii), and Theorem \ref{t:main} is proved in Section \ref{s:proof}. In Section \ref{s:example} we illustrate the whole argument on the exotic fusion system $\cF(3^4,9,2)$, which was the one case in \cite{LS25} requiring a separate treatment. Explicit bases for $\Ch(S)^\cF$ for the fusion systems on $p$-groups of order $p^4$, together with the case-by-case verification of Theorem \ref{t:main} for these systems which the present proof renders unnecessary, can be found in \cite{LS25}. \medskip \mbox{}\newline \textbf{Acknowledgements:} We thank Benjamin Sambale for making us aware of Olsson's work \cite{O82} and for the helpful discussions which followed, and John Shareshian for the argument in Proposition \ref{p:shareshian}. The proof of Theorem \ref{t:main} given in Sections \ref{s:bisets}--\ref{s:proof}, which was inspired by \cite[Conj. 5.5]{LS25}, and the reformulation of Theorem \ref{t:main} in Corollary \ref{cor:msform}, were found with the assistance of Claude, an AI model developed by Anthropic. J.~Semeraro gratefully acknowledges funding from the UK Research Council EPSRC for the project EP/W028794/1.

\section{The ring of $\cF$-stable virtual characters}\label{s:fstable}

Throughout, $\cl(S)$ denotes the set of $S$-conjugacy classes of $S$; for a positive integer $n$ we write $n_p$ for the $p$-part of $n$ and $v_p$ for the $p$-adic valuation. We shall use the following elementary bookkeeping constantly.

\begin{lem}\label{lem:hom-count}
Let $\cF$ be a fusion system on $S$ and $x\in S$. Then:
\begin{enumerate}
\item[(a)] the map $\Hom_\cF(\langle x\rangle,S)\to x^\cF$, $\varphi\mapsto\varphi(x)$, is a bijection; in particular $|\Hom_\cF(\langle x\rangle,S)|=|x^\cF|$;
\item[(b)] $\cF$-conjugacy of elements is an equivalence relation on $S$, and each $\cF$-class $x^\cF$ is a union of $S$-classes;
\item[(c)] $x$ is fully $\cF$-centralised if and only if the subgroup $\langle x\rangle$ is fully $\cF$-centralised in the sense of \cite[Part I, Def. 2.2]{AKO11}, and every $\cF$-class contains a fully $\cF$-centralised element.
\end{enumerate}
\end{lem}

\begin{proof}
(a) A homomorphism defined on the cyclic group $\langle x\rangle$ is determined by the image of the generator $x$, so the map is injective; it is surjective by the definition of $x^\cF$.

(b) Reflexivity is the presence of identity morphisms; transitivity follows from composition (if $\varphi\colon\langle x\rangle\to S$ and $\psi\colon\langle\varphi(x)\rangle\to S$ are $\cF$-morphisms then so is $\psi\circ\varphi$, after restricting the target). For symmetry, by the definition of a fusion system \cite[Part I, Def. 2.1]{AKO11} each $\varphi\in\Hom_\cF(\langle x\rangle,S)$ is the composite of an isomorphism $\langle x\rangle\to\langle\varphi(x)\rangle$ in $\cF$ and an inclusion, and the inverse of this isomorphism lies in $\cF$; hence $x\in\varphi(x)^\cF$. Finally, $\Hom_S(\langle x\rangle,S)\subseteq\Hom_\cF(\langle x\rangle,S)$ gives $x^S\subseteq x^\cF$, and $\cF$-conjugacy is $S$-conjugation invariant, so $x^\cF$ is a union of $S$-classes.

(c) By (a), the subgroups $\cF$-isomorphic to $\langle x\rangle$ are exactly the $\langle y\rangle$ with $y\in x^\cF$, and $C_S(\langle y\rangle)=C_S(y)$; so the two notions agree. Every $\cF$-class is finite, so it contains an element whose centraliser in $S$ has maximal order.
\end{proof}

\begin{lem}\label{lem:basisindep}
Let $\cF$ be a fusion system on $S$, let $B,B'$ be two bases of $\Ch(S)^\cF$, and write $\Gamma_B=(\langle\psi,\phi\rangle)_{\psi,\phi\in B}$ for the Gram matrix of $B$. Then
\[
\det\big(X_B\overline{X_B}^{\,T}\big)=\det\big(X_{B'}\overline{X_{B'}}^{\,T}\big)
\qquad\text{and}\qquad
\det\Gamma_B=\det\Gamma_{B'} .
\]
\end{lem}

\begin{proof}
$X_{B'}=MX_B$ for some $M\in \GL_n(\ZZ)$, so $X_{B'}\overline{X_{B'}}^{\,T}=MX_B\overline{X_B}^{\,T}M^T$ and the determinant is multiplied by $(\det M)^2=1$; similarly for $\Gamma$.
\end{proof}

\subsection{Indecomposable characters}
Recall that we write $\Ch(S)^\cF$ for the ring of virtual $\cF$-stable characters of $S$. We say that a character $\chi \in \Ch(S)^\cF$ is \textit{indecomposable} if whenever $\chi=\chi_1+\chi_2$ for a pair of characters $\chi_1,\chi_2 \in \Ch(S)^\cF$ we have $\chi=\chi_1$ or $\chi_2$. Following \cite{S24}, write $\Ind(\cF)$ for the set of indecomposable characters of $\Ch(S)^\cF$. The following is an elementary consequence of this definition.

\begin{lem}\label{l:ind}
We have $\langle \Ind(\cF) \rangle_\ZZ = \Ch(S)^\cF$.
\end{lem}

\begin{proof}
Every $\cF$-stable character is a sum of indecomposable ones, by induction on its degree. Every element of $\Ch(S)^\cF$ is a difference of two $\cF$-stable characters, since adding a large multiple of the regular character of $S$, which is $\cF$-stable, to a virtual $\cF$-stable character makes it a character.
\end{proof}

Recall that $\bk(\cF)$ denotes the number of $\cF$-conjugacy classes of elements of $S$. By \cite[Lem. 2.1]{BC20}, $\rk_\ZZ(\Ch(S)^\cF)=\bk(\cF)$, and in \cite[Thm. 1]{S24} it is shown that $\bk(\cF) \le |\Ind(\cF)|$. Using Brauer's characterisation of characters, one can easily establish the following:

\begin{prop}\label{prop:brauer}
Suppose $\cF=\cF_S(G)$ for some finite group $G$ with Sylow $p$-subgroup $S$. Then $\Ch(S)^\cF = \ZZ[\chi|_S \mid \chi \in \Irr(G)]$.
\end{prop}

\begin{proof}
See, for example, \cite[Thm. 4]{S24}.
\end{proof}

In the special case when $S \unlhd G$, we have the following stronger result, which is an elementary consequence of Clifford theory:

\begin{lem}\label{l:indfnormal}
Suppose $G$ is a finite group with a normal Sylow $p$-subgroup $S$ and let $\cF=\cF_S(G)$. Then $\Ind(\cF)=\{\chi|_S \mid \chi \in \Irr(G)\}.$ In particular, $|\Ind(\cF)|=\bk(\cF)$.
\end{lem}

\begin{proof}
See \cite[Thm. 2]{S24}.
\end{proof}

We close this section by showing that the conclusion of Theorem \ref{t:main} is false for non-saturated fusion systems.

\begin{exmp}\label{ex:c8}
Suppose $S=\langle \ba \rangle \cong C_8$, set $P:=\langle \ba^2 \rangle$ and $\cF:=\langle \cF_S(S); \Aut(P)\rangle_S$. Then $\cF$ is not saturated since the non-trivial element in $\Aut(P)$ which fuses $\ba^2$ to $\ba^6$ does not extend to $S$. Let $\omega:=\exp(2 \pi i / 8)$ and let $\chi_j \in \Irr(S)$ be given by $\chi_j(\ba^k):=\omega^{jk}$ for $0 \le j \le 7$. One easily calculates that: $$\Ind(\cF) = \{\chi_0,\chi_2,\chi_4,\chi_6, \chi_1+\chi_3, \chi_3+\chi_5,\chi_5+\chi_7, \chi_7+\chi_1\}$$ and so $\Ch(S)^\cF = \langle B \rangle$ where $B$ consists of the first $7$ of these characters by Lemma \ref{l:ind}. Now, $$\det(X_B(\cF)\overline{X_B(\cF)^T})=2^{22} \neq 2^{21} = \prod_{s \in \cF^z} |C_S(s)|. $$
\end{exmp}

\section{The group case}\label{s:groupcase}

We follow the approach taken in \cite{L24}, pointing out some overlap with that taken by Olsson in \cite{O82}. Let $G$ be a finite group, fix a (not necessarily Sylow) $p$-subgroup $S$ of $G$ and let $\cF=\cF_S(G)$ be the associated fusion system. Let $\Ch(S)=\ZZ[\Irr(S)]$ be the ring of virtual $S$-characters and $\Ch(S)^G:=\Ch(S)^{\cF_S(G)}$ be the subring consisting of those which are $G$-stable. Let $n=|S^G|$ be the number of $G$-classes of elements in $S$ and set $m=|\Irr(G)|$. Note that $\rk_\ZZ(\Ch(S)^G)=n$ by the discussion in Section \ref{s:fstable}. \medskip

Fix a basis $B$ of $\Ch(S)^G$ and let $\cF^z$ be a set of fully $\cF$-centralised $\cF$-conjugacy class representatives of $S$. Thus if $\rho$ is an $\cF$-stable character of $S$, we may write $$\rho=\sum_{\psi \in B} \lambda_{\rho,\psi} \psi$$ for unique $\lambda_{\rho,\psi} \in \ZZ$ depending on $B$. Let $D=D_B(G) \in M_{m \times n}(\ZZ)$ and $X=X_B(G) \in M_n(\CC)$ be the matrices given for $\chi \in \Irr(G),\psi \in B$ and $s \in \cF^z$ by: $$(D_B(G))_{\chi,\psi}:=\lambda_{\chi|_S,\psi} \mbox{ and } (X_B(G))_{\psi, s}:=\psi(s).$$ Since $(DX)_{\chi,s} = \displaystyle\sum_{\psi \in B} \lambda_{\chi|_S,\psi} \psi(s)=\chi(s),$
 \begin{equation}\label{e:main}
 (\overline{(DX)^T} DX)_{s,t} = \sum_{\chi \in \Irr(G)} \overline{\chi(s)}\chi(t) =\delta_{st}|C_G(s)|
 \end{equation} by column orthogonality. Thus writing $C:=\overline{D^T}D$, we have $$\det(\overline{(DX)^T}DX) = \det(C)\cdot \det(X\overline {X^T}),$$ and so \begin{equation}\label{e:2}
   \det(X\overline {X^T})=\frac{1}{\det(C)}\prod_{s \in \cF^z} |C_G(s)|.
\end{equation}

\begin{rem}\label{r:integral}
As was first noted in \cite[Cor. 3.18]{L24}, since $\det(C) \in \ZZ$, (\ref{e:2}) implies that $\det(X\overline {X^T})$ is a rational algebraic integer, and hence an integer. In fact, by considering the restriction to $S$ of the set of $\ell$-Brauer characters of $G$ for all primes $\ell \neq p$ one can moreover prove that $\det(X\overline{X^T}) \not\equiv 0 \pmod \ell$ whence $\det(X\overline{X^T})$ is in fact a $p$-power (see \cite[Section 4]{L24}).
\end{rem}

We assume for the remainder of this section that $S$ is a Sylow $p$-subgroup of $G$. The next result is essentially \cite[Folg. 6.9]{O82}; here we provide a self-contained proof.

\begin{prop}\label{p:groupcase}
If $S$ is a Sylow $p$-subgroup of $G$ then $(\det(C),p)=1$. Consequently, Theorem \ref{t:groupcase} holds.
\end{prop}

\begin{proof}
Writing $\Delta$ for the diagonal matrix $\Delta_{s,t} =\delta_{s,t}|C_G(s)|$ for $s,t \in \cF^z$, \eqref{e:main} implies $$\overline{X^T} C X = \overline{X^T}\overline{D^T}D X = \overline{(DX)^T}DX = \Delta$$ so that $C^{-1}=X \Delta^{-1}\overline{X^T}$ and explicitly for $\psi,\phi \in B$ we have, $$(C^{-1})_{\psi,\phi} = \sum_{s \in \cF^z} \frac{\psi(s)\overline{\phi(s)}}{|C_G(s)|}.$$ For $\psi \in B$ write $\psi_0$ for the $G$-class function given by $$\psi_0(g):=\begin{cases} |G|_{p'}\psi(s) & \mbox{ if $g$ is $G$-conjugate to $s \in S$ } \\ 0 & \mbox{ otherwise.} \end{cases}. $$ Then $\psi_0 \in \ZZ[\Irr(G)]$ by Brauer's characterisation of characters. Indeed, if $E$ is an elementary subgroup of $G$ (for any prime), then $E$ is nilpotent, so $E=O_p(E) \times O_{p'}(E)$, and replacing $E$ by a $G$-conjugate if necessary ($\psi_0$ being a class function of $G$) we may assume that $O_p(E) \le S$ since $S$ is a Sylow $p$-subgroup of $G$. Then for any $\rho \in \Irr(E)$, $$\langle \psi_0|_E,\rho \rangle = \frac{1}{|E|}\sum_{g \in E} \psi_0(g)\overline{\rho(g)} = \frac{1}{|E|}\sum_{g \in O_p(E)} |G|_{p'}\psi(g)\overline{\rho(g)} = \frac{|G|_{p'}}{|O_{p'}(E)|}\langle \psi|_{O_p(E)},\rho|_{O_p(E)} \rangle \in \ZZ.$$ Calculating, we obtain: $$|G|_{p'}^2 (C^{-1})_{\psi,\phi} = |G|_{p'}^2 \sum_{s \in \cF^z} \frac{\psi(s)\overline{\phi}(s)}{|C_G(s)|} = \sum_{s \in \cF^z} \frac{\psi_0(s)\overline{\phi_0(s)}}{|C_G(s)|} = \langle \psi_0,\phi_0 \rangle \in \ZZ.$$ Hence $|G|_{p'}^2C^{-1} \in M_n(\ZZ)$ from which it follows that $(\det(C),p)=1$.
Finally, for $\cF=\cF_S(G)$, $C_S(s) \in \Syl_p(C_G(s))$ if and only if $s$ is fully $\cF$-centralised. Thus by \eqref{e:2} we have $$\det(X\overline{X^T})_p=\prod_{s \in \cF^z} |C_G(s)|_p =\prod_{s \in \cF^z } |C_S(s)|,$$ as required.
\end{proof}

\subsection{Characters in the principal block}\label{s:groupcaseprinc}

There is an analogue $D_0$ of $D$ for the principal $p$-block $B_0$ of $G$ given by restricting its rows to $\Irr(B_0)$. Let $X_B(G)$ be as above and let $D_0=D_B(B_0)$ be given for $\chi \in \Irr(B_0)$, and $\psi \in B$ by $(D_B(B_0))_{\chi,\psi}:=\lambda_{\chi|_S,\psi}$ as before. Set $C_0=\overline{D_0^T}D_0$. Using orthogonality of characters in $B_0$, we show the following:

\begin{prop}\label{olssonblock}
 For $D_0=D_B(B_0)$ as above, we have:
 \begin{equation}\label{e:mainblock}
 (\overline{(D_0X)^T} D_0X)_{s,t} = \delta_{st}\dim(B_0(s))
 \end{equation}
 where $B_0(s)$ is the principal $p$-block of $C_G(s)$. Consequently $(\det(C_0),p)=1$.
\end{prop}

\begin{proof}
Since $(D_0X)_{\chi,s} = \displaystyle\sum_{\psi \in B} \lambda_{\chi|_S,\psi} \psi(s)=\chi(s),$

$$(\overline{(D_0X)^T} D_0X)_{s,t} = \sum_{\chi \in \Irr(B_0)} \overline{\chi(s)}\chi(t) =\delta_{st}\dim(B_0(s))$$
by a refinement of block orthogonality (see, e.g., \cite[Thm. 2.1]{KMS4}).
Now, $$\frac{1}{\det(C)}\prod_{s \in \cF^z} |C_G(s)|= \det(X\overline {X^T}) = \frac{1}{\det(C_0)}\prod_{s \in \cF^z} \dim(B_0(s)).$$ Since $(\det(C),p)=1$ by Proposition \ref{p:groupcase} and since $v_p(\dim(B_0(s)))=v_p(|C_G(s)|)$ for each $s \in \cF^z$ we must also have $(\det(C_0),p)=1$.
\end{proof}

\section{Bisets and the characteristic idempotent}\label{s:bisets}

In this section we recall what we need about bisets. The main point is Lemma~\ref{lem:kernel}, which expresses the action of a biset on class functions in terms of its permutation character; in Section~\ref{s:averaging} this will be combined with Reeh's Theorem~\ref{inp:reeh}.

\subsection{Bisets}
Let $S,T$ be finite groups. A \emph{$(T,S)$-biset} is a finite set $U$ with a left $T$-action and a right $S$-action that commute. It is \emph{left free} (resp.\ \emph{right free}) if the left (resp.\ right) action is free, and \emph{bifree} if both. Equivalently, $U$ is a $(T\times S)$-set via $(t,s)\cdot u=tus^{-1}$, and bifreeness says the point stabilisers intersect $T\times 1$ and $1\times S$ trivially.

Given $P\le S$ and an injective homomorphism $\varphi\colon P\to T$, set
\[
T\times_{(P,\varphi)}S:=(T\times S)/\!\sim,\qquad (t\varphi(u),s)\sim(t,us)\ \ (u\in P),
\]
with $T$ acting on the left factor and $S$ on the right; write $[t,s]$ for the class of $(t,s)$. This is a transitive bifree $(T,S)$-biset of cardinality $|T||S|/|P|$, and by Theorem~\ref{inp:bouc} every transitive bifree $(T,S)$-biset arises this way.

The \emph{double Burnside group} $A(T,S)$ is the Grothendieck group of bifree $(T,S)$-bisets under disjoint union; composition $(V,U)\mapsto V\times_T U$ makes $A(S,S)$ a ring with identity the regular biset $S$. Put $A(S,S)_{(p)}=A(S,S)\otimes\Zp$. The \emph{opposite} $U^{\mathrm{op}}$ of a $(T,S)$-biset $U$ is the $(S,T)$-biset with the same underlying set and $s\cdot u\cdot t:=t^{-1}us^{-1}$; one has $(S\times_{(P,\varphi)}T)^{\mathrm{op}}\cong T\times_{(\varphi(P),\varphi^{-1})}S$.

\subsection{The action on class functions}
Let $\cf(S)$ denote the $\CC$-space of class functions $S\to\CC$, with the inner product
\[
\langle f_1,f_2\rangle=\frac{1}{|S|}\sum_{g\in S}f_1(g)\overline{f_2(g)},
\]
and let $\CF(S)$ be the $\ZZ$-module of $\ZZ$-valued class functions on $S$. Recall the standard operations:
\begin{itemize}
\item restriction $\Res^S_P f=f|_P$;
\item transport along an isomorphism $\varphi\colon P\to Q$: $(\varphi_*f)(y)=f(\varphi^{-1}(y))$ for $y\in Q$;
\item induction, $\displaystyle \Ind_Q^S f(g)=\frac{1}{|Q|}\sum_{\substack{a\in S\\ g^a\in Q}} f(g^a)$.
\end{itemize}
The next lemma records that these operations preserve integrality. This simple fact is the reason why the elements of the $p$-local double Burnside ring preserve $\Zp$-valued class functions, which will be crucial in Section~\ref{sec:classes}.

\begin{lem}\label{lem:ind-integral}
For $Q\le S$ and $f\in\cf(Q)$,
\[
\Ind_Q^S f(g)=\sum_{\substack{t^Q\in\cl(Q)\\ t^Q\subseteq g^S\cap Q}}\frac{|C_S(g)|}{|C_Q(t)|}\,f(t),
\]
and each coefficient $|C_S(g)|/|C_Q(t)|$ is a positive integer. Consequently $\Res,\varphi_*$ and $\Ind$ all map $\ZZ$-valued class functions to $\ZZ$-valued class functions.
\end{lem}

\begin{proof}
Write $A=\{a\in S: g^a\in Q\}$, so that $\Ind_Q^Sf(g)=|Q|^{-1}\sum_{a\in A}f(g^a)$. Each $g^a$ with $a\in A$ lies in $g^S\cap Q$, and $f$ is constant on $Q$-classes, so grouping $A$ according to the $Q$-class of $g^a$ gives
\[
\Ind_Q^Sf(g)=\frac{1}{|Q|}\sum_{\substack{t^Q\in\cl(Q)\\ t^Q\subseteq g^S\cap Q}}\big|\{a\in A: g^a\in t^Q\}\big|\,f(t).
\]
Fix such a $t$. The map $S\to g^S$, $a\mapsto g^a$, is surjective with all fibres of size $|C_S(g)|$, so
\[
\big|\{a\in A: g^a\in t^Q\}\big|=|t^Q|\cdot|C_S(g)| ,
\]
and $|t^Q|=|Q|/|C_Q(t)|$. Dividing by $|Q|$ gives the displayed formula with coefficient $|C_S(g)|/|C_Q(t)|$. This is a positive integer: $t$ and $g$ are $S$-conjugate, so $|C_S(t)|=|C_S(g)|$, while $C_Q(t)=C_S(t)\cap Q\le C_S(t)$, whence $|C_Q(t)|$ divides $|C_S(t)|=|C_S(g)|$ by Lagrange's theorem. In particular, if $f$ is $\ZZ$-valued then so is $\Ind_Q^Sf$. For $\Res^S_P$ and $\varphi_*$ the corresponding statements are clear, since these operations merely restrict a class function to a subgroup, respectively relabel its arguments along an isomorphism.
\end{proof}

A bifree $(S,S)$-biset $U=\coprod_i S\times_{(P_i,\varphi_i)}S$ acts on $\cf(S)$ by
\begin{equation}\label{eq:Tdef}
T_U:=\sum_i \Ind_{\varphi_i(P_i)}^{S}\circ\, (\varphi_i)_*\circ \Res^{S}_{P_i} .
\end{equation}
This is the usual biset action (see \cite[Section 2.3]{B10} or \cite[Section 2]{RS13}), restricted from the action on representation rings; it is additive in $U$, multiplicative for composition of bisets, and extends $\Zp$-linearly to $A(S,S)_{(p)}$ acting on $\cf(S)$ and on $\CF(S)_{(p)}:=\CF(S)\otimes\Zp$. Since $\Ind$, $\Res$ and $\varphi_*$ send virtual characters to virtual characters, $T_\alpha$ also preserves $\Ch(S)_{(p)}:=\Ch(S)\otimes\Zp$ for every $\alpha\in A(S,S)_{(p)}$.

\subsection{Permutation characters}
For a bifree $(S,S)$-biset $U$, regarded as an $(S\times S)$-set via $(g,h)\cdot u=guh^{-1}$, the permutation character is
\[
\chi_U(g,h)=\big|U^{(g,h)}\big|=\big|\{u\in U : gu=uh\}\big| ,
\]
extended $\Zp$-linearly to $A(S,S)_{(p)}$. Clearly $\chi_{U^{\mathrm{op}}}(g,h)=\chi_U(h,g)$.

The next lemma shows that $T_\alpha$ is determined by the function $\chi_\alpha$ on $S\times S$: for a class function $f$, the value of $T_\alpha f$ at $g$ is the average over $h\in S$ of $\chi_\alpha(g,h)f(h)$. In other words, $T_\alpha$ acts through the matrix $\big(|S|^{-1}\chi_\alpha(g,h)\big)_{g,h\in S}$, whose rows and columns are indexed by the elements of $S$. As is customary for operators given by a formula of this kind, we call $\chi_\alpha$ the \emph{kernel} of $T_\alpha$; this should not be confused with the null space of $T_\alpha$.

\begin{lem}[Kernel formula]\label{lem:kernel}
Let $\alpha\in A(S,S)_{(p)}$. Then for all $f\in\cf(S)$ and $g\in S$,
\[
(T_\alpha f)(g)=\frac{1}{|S|}\sum_{h\in S}\chi_\alpha(g,h)\,f(h).
\]
\end{lem}

\begin{proof}
Both sides are $\Zp$-linear in $\alpha$, and $A(S,S)_{(p)}$ is spanned over $\Zp$ by the classes of transitive bifree bisets. So it suffices to treat $\alpha=[U]$ with $U=S\times_{(P,\varphi)}S$ transitive bifree. Its elements are the classes $[a,b]$ of pairs $(a,b)\in S\times S$, subject to $[a\varphi(u),b]=[a,ub]$ for $u\in P$. Equivalently, the representatives of $[a,b]$ are exactly the $|P|$ pairs
\[
(a\varphi(u),\,u^{-1}b)\qquad(u\in P).
\]

\emph{The permutation character.} We have $g[a,b]=[a,b]h$ if and only if $[ga,b]=[a,bh]$. By the description of representatives just given, this holds if and only if there exists $u\in P$ with
\[
ga=a\varphi(u)\qquad\text{and}\qquad b=u^{-1}bh,
\]
that is,
\[
\varphi(u)=a^{-1}ga=g^{a}\qquad\text{and}\qquad u=bhb^{-1}=h^{b^{-1}} .
\]
Since $\varphi$ is injective, such a $u$ is uniquely determined by $a$, namely $u=\varphi^{-1}(g^a)$. Dividing the number of such pairs $(a,b)$ by the number $|P|$ of representatives of each element of $U$,
\[
\chi_U(g,h)=\frac{1}{|P|}\,\Big|\big\{(a,b)\in S\times S:\ g^{a}\in\varphi(P)\ \text{and}\ h^{b^{-1}}=\varphi^{-1}(g^a)\big\}\Big| .
\]

\emph{The right-hand side.} Fix $a\in S$ with $g^a\in\varphi(P)$ and put $u_a=\varphi^{-1}(g^a)\in P$. For each $b\in S$ there is exactly one $h\in S$ with $h^{b^{-1}}=u_a$, namely $h=b^{-1}u_ab=u_a^{\,b}$. For this $h$ we have $f(h)=f(u_a)$, because $f$ is a class function. Summing over $h$ therefore counts, for each such $a$, exactly $|S|$ pairs $(b,h)$, each contributing $f(u_a)$. Hence
\[
\frac{1}{|S|}\sum_{h\in S}\chi_U(g,h)f(h)
=\frac{1}{|S|\,|P|}\cdot|S|\sum_{\substack{a\in S\\ g^a\in\varphi(P)}} f\big(\varphi^{-1}(g^a)\big)
\]
\[
=\frac{1}{|\varphi(P)|}\sum_{\substack{a\in S\\ g^a\in\varphi(P)}}\big((\varphi)_*\Res^S_Pf\big)(g^a),
\]
using $|\varphi(P)|=|P|$ and $\big((\varphi)_*\Res^S_Pf\big)(y)=f(\varphi^{-1}(y))$ for $y\in\varphi(P)$. By the definition of induction, the last expression is
\[
\big(\Ind_{\varphi(P)}^S(\varphi)_*\Res^S_P f\big)(g)=(T_Uf)(g).
\]
\end{proof}

\begin{rem}[Conventions]\label{rem:convention}
Two conventions could be questioned in Lemma~\ref{lem:kernel}: whether \eqref{eq:Tdef} or its opposite is the ``correct'' action of $S\times_{(P,\varphi)}S$, and whether $\chi_U(g,h)$ should be read as $|\{u:gu=uh\}|$ or $|\{u:uh=gu\}|$ with the roles of the two arguments interchanged. Replacing $U$ by $U^{\mathrm{op}}$ interchanges the two arguments of $\chi_U$, so the effect of either change is to transpose the kernel. This will be immaterial below: the kernel $\chi_{\omega_\cF}$ we apply the lemma to is symmetric (Proposition~\ref{prop:average}), so both conventions give the same operator $\pi$. As a sanity check of normalisation, take $U=S$, the identity of $A(S,S)$: then $\chi_S(g,h)=|\{x\in S: x^{-1}gx=h\}|$, which is $|C_S(g)|$ if $h\in g^S$ and $0$ otherwise, so the right-hand side of Lemma~\ref{lem:kernel} is $|S|^{-1}|g^S|\,|C_S(g)|f(g)=f(g)$, as it must be.
\end{rem}

\begin{rem}\label{rem:bouc-lite}
Theorem~\ref{inp:bouc} is used only to write a general bifree biset as a disjoint union of the standard transitive ones, so that \eqref{eq:Tdef} makes sense as stated and Lemma~\ref{lem:kernel} can be proved one orbit at a time.
\end{rem}

\subsection{Characteristic bisets and idempotents}
The following conditions go back to Linckelmann and Webb (see \cite[Section 4]{RS13}, where they are discussed in detail, or \cite[Section 3]{R06}). A bifree $(S,S)$-biset $\Omega$ is \emph{$\cF$-characteristic} if
\begin{enumerate}
\item[(C1)] $\Omega$ is \emph{$\cF$-generated}: each transitive component is of the form $S\times_{(P,\varphi)}S$ with $\varphi\in\Hom_\cF(P,S)$;
\item[(C2)] $\Omega$ is \emph{$\cF$-stable}: for every $P\le S$ and $\varphi\in\Hom_\cF(P,S)$, the $(S,P)$-bisets obtained from $\Omega$ by restricting the right action along $\varphi$ and along the inclusion $P\hookrightarrow S$ are isomorphic, and similarly for the left action;
\item[(C3)] $|\Omega|/|S|\not\equiv 0\pmod p$.
\end{enumerate}
Puig proved that a saturated fusion system admits an $\cF$-characteristic biset; conversely the existence of such a biset characterises saturation \cite[Thm. A]{RS13}. Ragnarsson \cite{R06} refined this to the idempotent level; see Theorem~\ref{inp:rag}. We denote by $\omega_\cF\in A(S,S)_{(p)}$ the characteristic idempotent and write
\[
\pi:=T_{\omega_\cF}\colon \cf(S)\to\cf(S).
\]
By the discussion after \eqref{eq:Tdef}, $\pi$ preserves both $\CF(S)_{(p)}$ and $\Ch(S)_{(p)}$. This innocuous remark is the crux of Section~\ref{sec:classes}.

\begin{rem}
The usual route from $\omega_\cF$ to $\cF$-stable elements is the \emph{universal stable element theorem} \cite[Thm. 4.12]{RS13}: for a $p$-local $\cF$-defined Mackey functor $M$, an element $x\in M(S)$ is $\cF$-stable if and only if $M(\omega_\cF)(x)=x$; for representation rings this is developed in \cite{RY18}. We will not need it: Proposition~\ref{prop:average} gives the required statements for class functions directly and by hand.
\end{rem}

\section{The averaging formula}\label{s:averaging}

We now combine Lemma~\ref{lem:kernel} with Reeh's Theorem~\ref{inp:reeh}. The resulting description of $\pi=T_{\omega_\cF}$ as an averaging operator is the basis of everything that follows.

\begin{prop}\label{prop:average}
Let $\cF$ be a saturated fusion system on $S$ and $\pi=T_{\omega_\cF}$. Then for all $f\in\cf(S)$ and $g\in S$,
\[
(\pi f)(g)=\frac{1}{|g^{\cF}|}\sum_{h\in g^{\cF}} f(h).
\]
Consequently:
\begin{enumerate}
\item[(a)] $\pi$ is the orthogonal projection of $\cf(S)$ onto the subspace $\cf(S)^\cF$ of $\cF$-stable class functions; in particular $\pi^2=\pi$, $\operatorname{im}\pi=\cf(S)^\cF$ and $\pi|_{\cf(S)^\cF}=\mathrm{id}$;
\item[(b)] $\pi$ is self-adjoint: $\langle\pi f_1,f_2\rangle=\langle f_1,\pi f_2\rangle$ for all $f_1,f_2\in\cf(S)$;
\item[(c)] for every $x\in S$,
\[
\pi\big(1_{x^S}\big)=\frac{|x^{S}|}{|x^{\cF}|}\cdot 1_{x^{\cF}},
\]
where $1_Y$ denotes the indicator function of $Y\subseteq S$;
\item[(d)] $\pi\big(\Ch(S)_{(p)}\big)=\Ch(S)^{\cF}\otimes\Zp$ and $\pi\big(\CF(S)_{(p)}\big)\subseteq\CF(S)_{(p)}$.
\end{enumerate}
\end{prop}

\begin{proof}
By Theorem~\ref{inp:reeh} together with Lemma~\ref{lem:hom-count}(a), the kernel $\chi_{\omega_\cF}$ of $\pi$, in the sense of Lemma~\ref{lem:kernel}, is
\[
\chi_{\omega_\cF}(g,h)=\begin{cases}|S|/|g^\cF| & h\in g^\cF,\\ 0&\text{otherwise,}\end{cases}
\]
which is symmetric in $(g,h)$: indeed $h\in g^\cF\iff g\in h^\cF$ by Lemma~\ref{lem:hom-count}(b), and in that case $g^\cF=h^\cF$, so the two values agree. Lemma~\ref{lem:kernel} now gives
\[
(\pi f)(g)=\frac{1}{|S|}\sum_{h\in g^\cF}\frac{|S|}{|g^\cF|}f(h)=\frac{1}{|g^\cF|}\sum_{h\in g^\cF}f(h).
\]

(a) and (b). For $s\in\Fz$ write $1_{s^\cF}$ for the indicator function of the $\cF$-class $s^\cF$, and set $\sigma_s(f)=\sum_{x\in s^\cF}f(x)$. A class function is $\cF$-stable precisely when it is constant on each $\cF$-class, so the $1_{s^\cF}$, $s\in\Fz$, span $\cf(S)^\cF$; they have pairwise disjoint supports, hence are an \emph{orthogonal} basis of it, and
\[
\langle f,1_{s^\cF}\rangle=\frac{\sigma_s(f)}{|S|}\quad (f\in\cf(S)),\qquad
\langle 1_{s^\cF},1_{s^\cF}\rangle=\frac{|s^\cF|}{|S|} .
\]
Consequently the orthogonal projection of $\cf(S)$ onto $\cf(S)^\cF$ is
\[
f\ \longmapsto\ \sum_{s\in\Fz}\frac{\langle f,1_{s^\cF}\rangle}{\langle 1_{s^\cF},1_{s^\cF}\rangle}\,1_{s^\cF}
=\sum_{s\in\Fz}\frac{\sigma_s(f)}{|s^\cF|}\,1_{s^\cF},
\]
whose value at $g$ is the average of $f$ over $g^\cF$. By the displayed formula above, this map is exactly $\pi$. So $\pi$ is the orthogonal projection onto $\cf(S)^\cF$, which gives (a); and an orthogonal projection is self-adjoint, which gives (b). Explicitly, grouping the sum over $S$ by $\cF$-classes,
\[
\langle\pi f_1,f_2\rangle=\frac{1}{|S|}\sum_{s\in\Fz}\frac{\sigma_s(f_1)\,\overline{\sigma_s(f_2)}}{|s^\cF|}
=\overline{\frac{1}{|S|}\sum_{s\in\Fz}\frac{\sigma_s(f_2)\,\overline{\sigma_s(f_1)}}{|s^\cF|}}
=\overline{\langle\pi f_2,f_1\rangle}=\langle f_1,\pi f_2\rangle ,
\]
the last equality because the form is Hermitian.

(c) If $g\notin x^\cF$ then $g^\cF\cap x^S=\emptyset$ and the average vanishes. If $g\in x^\cF$ then $g^\cF=x^\cF\supseteq x^S$ and the average is $|x^S|/|x^\cF|$.

(d) $\pi$ preserves $\Ch(S)_{(p)}$ and $\CF(S)_{(p)}$ because $\omega_\cF\in A(S,S)_{(p)}$ and $\Ind,\Res,\varphi_*$ preserve virtual characters and $\ZZ$-valued class functions (Lemma~\ref{lem:ind-integral}). By (a), $\pi\big(\Ch(S)_{(p)}\big)\subseteq \Ch(S)_{(p)}\cap\cf(S)^\cF=\Ch(S)^\cF\otimes\Zp$, the last equality because $\Ch(S)^\cF$ is pure in $\Ch(S)$ (if $n\chi$ is constant on $\cF$-classes with $n\neq0$, so is $\chi$). Conversely $\pi$ fixes $\Ch(S)^\cF\otimes\Zp$ pointwise by (a), so the inclusion is an equality.
\end{proof}

Two remarks on part (b). First, the form $\langle\cdot,\cdot\rangle$ on $\cf(S)$ is Hermitian rather than bilinear, so self-adjointness is the identity $\langle\pi f_1,f_2\rangle=\langle f_1,\pi f_2\rangle$ displayed above, and not a statement about a transpose. Second, Proposition~\ref{prop:selfdual} applies (b) to the symmetric $\ZZ$-valued bilinear form on $\Ch(S)$ of Lemma~\ref{lem:unimodular}, not to the Hermitian form on $\cf(S)$. This is legitimate: for $a,b\in\Ch(S)_{(p)}$ the three numbers $\langle a,b\rangle$, $\langle\pi a,b\rangle$ and $\langle a,\pi b\rangle$ all lie in $\Zp\subseteq\Q$ by (d), so complex conjugation acts trivially on them and the Hermitian identity of (b) is an identity of symmetric $\Zp$-bilinear forms on $\Ch(S)_{(p)}$, which is what Lemma~\ref{lem:dual} requires.

We pause for three consistency checks on Proposition~\ref{prop:average}.

\begin{exmp}
If $\cF=\cF_S(S)$ then $\omega_\cF$ is the identity biset $S$ and $\cF$-classes are $S$-classes; the formula says that averaging a class function over an $S$-class is the identity operator, which is correct, and agrees with Remark~\ref{rem:convention}.
\end{exmp}

\begin{exmp}
$\cF$ is \emph{transitive} if $\Fz=\{1,z\}$, i.e.\ all non-identity elements of $S$ are $\cF$-conjugate. Then $\pi f$ takes the value $f(1)$ at $1$ and the value $(|S|-1)^{-1}\sum_{g\ne1}f(g)$ elsewhere, which is visibly an idempotent operator with image the $2$-dimensional space of $\cF$-stable class functions.
\end{exmp}

\begin{exmp}[Trace check]\label{ex:trace}
The functions $\{1_{t^S}\}_{t^S\in\cl(S)}$ form a basis of $\cf(S)$, and by Proposition~\ref{prop:average}(c) the matrix of $\pi$ in this basis has diagonal entries $|t^S|/|t^\cF|$. Hence
\[
\tr(\pi)=\sum_{t^S\in\cl(S)}\frac{|t^S|}{|t^\cF|}=\sum_{s\in\Fz}\frac{1}{|s^\cF|}\sum_{t^S\subseteq s^\cF}|t^S|=\sum_{s\in\Fz}1=\bk(\cF),
\]
consistent with $\pi$ being a projection onto a space of dimension $\bk(\cF)$.
\end{exmp}

\begin{cor}\label{cor:rank}
For $\cF$ saturated, $\Ch(S)^\cF$ is a lattice of rank $\bk(\cF)$, and its $\CC$-span is $\cf(S)^\cF$.
\end{cor}

\begin{proof}
$\Ch(S)$ spans $\cf(S)$ over $\CC$, so $\pi(\Ch(S)_{(p)})$ spans $\pi(\cf(S))=\cf(S)^\cF$, of dimension $\bk(\cF)$. By Proposition~\ref{prop:average}(d) this span is contained in the $\CC$-span of $\Ch(S)^\cF$, which in turn lies in $\cf(S)^\cF$.
\end{proof}

\section{Self-duality of $\Ch(S)^{\cF}$ at $p$}\label{sec:selfdual}

We now prove the first of the two ingredients in the proof of Theorem~\ref{t:main}, namely Proposition~\ref{prop:selfdual}: for every basis $B$ of $\Ch(S)^\cF$ the Gram determinant $\det\Gamma_B$ is prime to $p$. Section~\ref{sec:forms} collects the facts about lattices that we need; they are standard, and we refer to \cite[Chapter I]{MH73} and \cite[Chapter I]{C96} for more details. Section~\ref{ss:proj} shows that $\Ch(S)$ is self-dual and describes the orthogonal projection of a self-dual lattice onto a subspace, and Section~\ref{ss:selfdual} combines this with the averaging formula.

\subsection{Lattices and their duals}\label{sec:forms}

Throughout this section $R$ denotes $\ZZ$ or $\Zp$. Both are principal ideal domains with field of fractions $\Q$; we have $\ZZ^\times=\{\pm1\}$, while $\Zp^\times$ consists of the fractions $a/b$ with $a,b\in\ZZ$ both prime to $p$. In particular an integer is a unit of $\Zp$ if and only if it is prime to $p$.

By an \emph{$R$-lattice} we mean a free $R$-module $N$ of finite rank $n$ equipped with a symmetric bilinear form $\langle\cdot,\cdot\rangle\colon N\times N\to R$; the form extends to the $\Q$-vector space $N_\Q:=N\otimes_R\Q$. The lattices of interest are $\Ch(S)$, its sublattice $\Ch(S)^\cF$, and their localisations $\Ch(S)_{(p)}$ and $\Ch(S)^\cF\otimes\Zp$, with the form given by the inner product of class functions (see Lemma~\ref{lem:unimodular}).

The \emph{Gram matrix} of a basis $C=\{c_1,\dots,c_n\}$ of $N$ is $\Gamma_C=(\langle c_i,c_j\rangle)_{i,j}\in M_n(R)$. Changing the basis by a matrix $M\in\GL_n(R)$ replaces $\Gamma_C$ by $M\Gamma_CM^T$, and so multiplies $\det\Gamma_C$ by the square of a unit of $R$. Hence whether $\det\Gamma_C$ is nonzero, or is a unit of $R$, depends only on $N$. The form is \emph{nondegenerate} if $\det\Gamma_C\ne0$, and in that case the \emph{dual lattice} of $N$ is
\[
N^\#=\{v\in N_\Q:\ \langle v,\lambda\rangle\in R\ \text{for all}\ \lambda\in N\}.
\]
We call $N$ \emph{self-dual} (or \emph{unimodular}) if $N^\#=N$. Equivalently, the map $N\to\Hom_R(N,R)$, $\lambda\mapsto\langle\lambda,\cdot\rangle$, is an isomorphism; its matrix with respect to $C$ and the dual basis of $\Hom_R(N,R)$ is $\Gamma_C$.

\begin{lem}\label{l:duallattice}
Let $N$ be an $R$-lattice with nondegenerate form and basis $C=\{c_1,\dots,c_n\}$.
\begin{enumerate}
\item[(a)] $N^\#$ is a free $R$-module with basis $C^\#=\{c_1^\#,\dots,c_n^\#\}$, where
\[
c_i^\#=\sum_{j=1}^n(\Gamma_C^{-1})_{ij}\,c_j ;
\]
these are the unique elements of $N_\Q$ with $\langle c_i^\#,c_k\rangle=\delta_{ik}$. In particular $N\subseteq N^\#$.
\item[(b)] $N$ is self-dual if and only if $\Gamma_C^{-1}\in M_n(R)$, if and only if $\det\Gamma_C\in R^\times$.
\item[(c)] If $R=\ZZ$, then $N^\#/N$ is a finite group of order $|\det\Gamma_C|$. Moreover $(N\otimes\Zp)^\#=N^\#\otimes\Zp$, and $N\otimes\Zp$ is self-dual if and only if $p\nmid\det\Gamma_C$.
\end{enumerate}
\end{lem}

\begin{proof}
(a) We have $\langle c_i^\#,c_k\rangle=\sum_j(\Gamma_C^{-1})_{ij}(\Gamma_C)_{jk}=\delta_{ik}$, so $c_i^\#\in N^\#$, and these conditions determine $c_i^\#$ by nondegeneracy. Conversely, each $v\in N^\#$ pairs with every $c_k$ in the same way as $\sum_i\langle v,c_i\rangle\,c_i^\#$, so the two are equal; since $\langle v,c_i\rangle\in R$, the $c_i^\#$ span $N^\#$ over $R$. They are linearly independent because $\Gamma_C^{-1}$ is invertible.

(b) By (a), $N^\#=N$ if and only if every $c_i^\#$ lies in $N$, that is, $\Gamma_C^{-1}\in M_n(R)$. If this holds then $\det\Gamma_C\cdot\det\Gamma_C^{-1}=1$ with both factors in $R$, so $\det\Gamma_C\in R^\times$. Conversely, if $\det\Gamma_C\in R^\times$ then $\Gamma_C^{-1}=(\det\Gamma_C)^{-1}\operatorname{adj}(\Gamma_C)$ has entries in $R$.

(c) By (a), $c_j=\sum_i(\Gamma_C)_{ji}\,c_i^\#$. Thus, with respect to the bases $C$ and $C^\#$, the inclusion $N\subseteq N^\#$ is given by the integer matrix $\Gamma_C$, and putting $\Gamma_C$ in Smith normal form shows that $[N^\#:N]=|\det\Gamma_C|$. The basis $C$ of $N$ is also a $\Zp$-basis of $N\otimes\Zp$ with the same Gram matrix, so by (a) the dual of $N\otimes\Zp$ has the same basis $C^\#$ as $N^\#$; the last assertion is (b) for $R=\Zp$.
\end{proof}

We shall apply Lemma~\ref{l:duallattice} to sublattices $N$ of a larger lattice $L$. Such an $N$ will again be an $R$-lattice, being a submodule of a free module over a principal ideal domain, and its dual is formed inside $N_\Q$, a subspace of $L_\Q$.

\begin{exmp}\label{ex:gram}
For $\cF=\cF(3^4,9,2)$ and the basis $B$ of $\Ch(S)^\cF$ given in Section~\ref{s:example}, one has $\det\Gamma_B=608=2^5\cdot 19$. By Lemma~\ref{l:duallattice}(c), $\Ch(S)^\cF$ has index $608$ in its dual lattice, so it is not self-dual over $\ZZ$; it is, however, self-dual at $p=3$, as Proposition~\ref{prop:selfdual} below requires. By contrast $\Ch(S)$ has Gram matrix the identity with respect to $\Irr(S)$, so it is self-dual over $\ZZ$.
\end{exmp}

\subsection{Virtual characters and projections of self-dual lattices}\label{ss:proj}

We now turn to the lattices of interest, starting with $\Ch(S)$.

\begin{lem}\label{lem:unimodular}
The inner product of class functions restricts to a symmetric $\ZZ$-valued bilinear form on $\Ch(S)$, with respect to which $\Irr(S)$ is an orthonormal basis; in particular $\Ch(S)$ and $\Ch(S)_{(p)}$ are self-dual. The form is positive definite on $\Ch(S)_\Q$, hence nondegenerate on every subspace of $\Ch(S)_\Q$, and $\det\Gamma_B>0$ for every basis $B$ of $\Ch(S)^\cF$.
\end{lem}

\begin{proof}
For virtual characters $\psi,\phi$ the number $\langle\psi,\phi\rangle$ is an integer, so $\langle\psi,\phi\rangle=\overline{\langle\phi,\psi\rangle}=\langle\phi,\psi\rangle$; and the Hermitian form is bilinear on $\Ch(S)$ because complex conjugation fixes integer scalars. The first orthogonality relation says that the Gram matrix of $\Irr(S)$ is the identity, so $\Ch(S)$ and $\Ch(S)_{(p)}$ are self-dual by Lemma~\ref{l:duallattice}(b), and $\langle\cdot,\cdot\rangle$ is positive definite on $\Ch(S)_\Q$. A positive definite form is nondegenerate on every subspace, and the Gram matrix of any basis of a subspace has positive determinant.
\end{proof}

As Example~\ref{ex:gram} shows, the sublattice $\Ch(S)^\cF$ need not be self-dual over $\ZZ$. The following lemma, which is the only genuinely lattice-theoretic input in the proof, describes the dual of a sublattice of a self-dual lattice in terms of an orthogonal projection.

\begin{lem}\label{lem:dual}
Let $L$ be a self-dual $\Zp$-lattice, let $W$ be a subspace of $L_\Q$ on which the form is nondegenerate, and let $\pi_W\colon L_\Q\to W$ be the orthogonal projection. Then
\[
\pi_W(L)=(L\cap W)^{\#},
\]
the dual of $L\cap W$ being formed inside $W$.
\end{lem}

\begin{proof}
First, $L\cap W$ spans $W$, since every $w\in W$ has a nonzero integer multiple in $L$. For $\lambda\in L$ and $\mu\in L\cap W$ we have $\langle\pi_W\lambda,\mu\rangle=\langle\lambda,\mu\rangle\in\Zp$, because $\lambda-\pi_W\lambda$ is orthogonal to $W$; hence $\pi_W(L)\subseteq (L\cap W)^{\#}$.

Conversely, let $v\in (L\cap W)^{\#}$. The submodule $L\cap W$ is pure in $L$ (if $n\lambda\in W$ for some integer $n\ne0$ then $\lambda\in W$), hence a direct summand of $L$, so every $\Zp$-linear functional on $L\cap W$ extends to $L$. Applying this to the functional $\mu\mapsto\langle\mu,v\rangle$, and using that $L$ is self-dual, we obtain $\lambda\in L$ with $\langle\mu,v\rangle=\langle\mu,\lambda\rangle$ for all $\mu\in L\cap W$. Then $v-\pi_W\lambda\in W$ is orthogonal to $L\cap W$, which spans $W$, so $v=\pi_W\lambda$ by nondegeneracy.
\end{proof}

In general $L\cap W\subseteq(L\cap W)^\#=\pi_W(L)$, and by Lemma~\ref{l:duallattice}(b) equality holds exactly when the Gram determinant of $L\cap W$ is a unit. The content of Proposition~\ref{prop:selfdual} below is that for the projection onto $\cF$-stable class functions equality is forced, because this projection is induced by the element $\omega_\cF$ of the $p$-local double Burnside ring.

\subsection{Self-duality of $\Ch(S)^{\cF}_{(p)}$}\label{ss:selfdual}

We can now prove the main result of this section.

\begin{prop}\label{prop:selfdual}
Let $\cF$ be a saturated fusion system on $S$ and let $B$ be a $\ZZ$-basis of $\Ch(S)^\cF$. Then $\Ch(S)^\cF\otimes\Zp$ is self-dual; equivalently,
\[
p\nmid\det\Gamma_B,\qquad\text{where }\ \Gamma_B=\big(\langle\psi,\phi\rangle\big)_{\psi,\phi\in B}.
\]
\end{prop}

\begin{proof}
Put $L=\Ch(S)_{(p)}$ and $M=\Ch(S)^\cF\otimes\Zp$, and let $W=M_\Q$ be the $\Q$-span of $M$ in $L_\Q$. By Lemma~\ref{lem:unimodular}, $L$ is self-dual and the form is nondegenerate on $W$. Moreover $L\cap W=M$, because $\Ch(S)^\cF$ is pure in $\Ch(S)$: if $n\chi$ is $\cF$-stable for some integer $n\ne0$, then so is $\chi$.

By Proposition~\ref{prop:average}(d) we have $\pi(L)=M$, so $\pi$ maps $L_\Q$ onto $W$. By Proposition~\ref{prop:average}(a),(b), $\pi$ is idempotent and self-adjoint, so its restriction to $L_\Q$ is the orthogonal projection $\pi_W$ onto $W$. Lemma~\ref{lem:dual} now gives
\[
M=\pi(L)=\pi_W(L)=(L\cap W)^{\#}=M^{\#},
\]
so $M$ is self-dual. Since $B$ is also a $\Zp$-basis of $M$, Lemma~\ref{l:duallattice}(b) shows that $\det\Gamma_B\in\Zp^\times$, that is, $p\nmid\det\Gamma_B$.
\end{proof}

\section{The $p$-part of an $\cF$-class}\label{sec:classes}

We now prove the second ingredient, Theorem~\ref{thm:classsize}, which is \cite[Conj.~5.5]{LS25}. It uses parts (c) and (d) of Proposition~\ref{prop:average} and nothing from Section~\ref{sec:selfdual}.

\begin{lem}\label{lem:split}
Let $\cF$ be a fusion system on $S$, let $x\in S$ be fully $\cF$-centralised, and write $x^\cF=\coprod_i t_i^S$ with $t_1=x$. Then each $|C_S(x)|/|C_S(t_i)|$ is a power of $p$ (possibly $1$), and
\[
\frac{|x^\cF|}{|x^S|}=\sum_i\frac{|C_S(x)|}{|C_S(t_i)|}=:m_x\in\Z_{>0}.
\]
Moreover $m_x\equiv \#\{i : t_i \text{ is fully } \cF\text{-centralised}\} \pmod p$.
\end{lem}

\begin{proof}
By Lemma~\ref{lem:hom-count}(b) the decomposition exists. Since $x$ is fully $\cF$-centralised, $|C_S(t_i)|\le|C_S(x)|$, and both are powers of $p$, so the quotient is a power of $p$, equal to $1$ precisely when $t_i$ is itself fully $\cF$-centralised. Now
\[
|x^\cF|=\sum_i |t_i^S|=\sum_i [S:C_S(t_i)]=[S:C_S(x)]\sum_i\frac{|C_S(x)|}{|C_S(t_i)|}=|x^S|\,m_x,
\]
and reducing $m_x$ modulo $p$ kills all terms with $|C_S(t_i)|<|C_S(x)|$.
\end{proof}

In particular $|x^S|$ divides $|x^\cF|$. Theorem~\ref{thm:classsize} below says that $|x^S|$ is exactly the $p$-part of $|x^\cF|$. The upper bound for $|x^\cF|_p$ comes from the integrality of the averaging operator $\pi$:

\begin{cor}\label{cor:upper}
Let $\cF$ be a saturated fusion system on $S$ and let $x\in S$. Then $|x^{\cF}|_p$ divides $|x^{S}|$; equivalently $|C_S(x)|\cdot|x^\cF|_p\le |S|$.
\end{cor}

\begin{proof}
The indicator function $1_{x^S}$ lies in $\CF(S)_{(p)}$, so $\pi(1_{x^S})\in\CF(S)_{(p)}$ by Proposition~\ref{prop:average}(d). By Proposition~\ref{prop:average}(c) its value at $x$ is $|x^S|/|x^\cF|$, so this rational number lies in $\Zp$, i.e.\ $v_p(|x^\cF|)\le v_p(|x^S|)$. Both $|x^\cF|_p$ and $|x^S|$ are powers of $p$, so the first assertion follows; the second is a restatement, using $|x^S|=[S:C_S(x)]$.
\end{proof}

The lower bound is elementary, and holds for any fusion system:

\begin{lem}\label{lem:lower}
Let $\cF$ be a fusion system on $S$ and $x\in S$ fully $\cF$-centralised. Then $|x^S|$ divides $|x^\cF|$, and hence divides $|x^\cF|_p$.
\end{lem}

\begin{proof}
This is Lemma~\ref{lem:split}: $|x^\cF|=|x^S|m_x$ with $m_x$ a positive integer. Since $|x^S|$ is a power of $p$ dividing $|x^\cF|$, it divides $|x^\cF|_p$.
\end{proof}

\begin{thm}\label{thm:classsize}
Let $\cF$ be a saturated fusion system on $S$ and let $x\in S$ be fully $\cF$-centralised. Then
\[
|x^{\cF}|_p=|x^S|=[S:C_S(x)],\qquad\text{equivalently}\qquad |C_S(x)|\cdot\big|\Hom_{\cF}(\langle x\rangle,S)\big|_p=|S| .
\]
Equivalently again, the integer $m_x$ of Lemma~\ref{lem:split} is prime to $p$; in particular the number of $S$-classes of fully $\cF$-centralised elements inside $x^\cF$ is prime to $p$.
\end{thm}

\begin{proof}
Combine Corollary~\ref{cor:upper} and Lemma~\ref{lem:lower}; the reformulations use Lemma~\ref{lem:hom-count}(a), $|x^\cF|=|x^S|m_x$, and the last congruence of Lemma~\ref{lem:split}.
\end{proof}

In the terminology of \cite[Def. 5.1]{LS25}, Theorem~\ref{thm:classsize} says that the function $\chi_{\omega_\cF}$ of Theorem~\ref{inp:reeh} is \emph{$\cF$-good}; this is \cite[Conj. 5.5]{LS25}, and by Corollary~\ref{cor:equiv} below it is in fact equivalent to Theorem~\ref{t:main}.

\begin{rem}
Corollary~\ref{cor:upper} reproves, for cyclic subgroups, the inequality $|C_S(x)|\,|\Hom_\cF(\lb x\rb,S)|_p\le|S|$ of Boltje--Danz \cite[Cor. 7.17]{BD12} quoted in \cite[Rem. 5.8]{LS25}. The point of Theorem~\ref{thm:classsize} is that the opposite divisibility is elementary, so that the two together give the equality described as elusive there.
\end{rem}

\begin{rem}[The realisable case]\label{r:realisable}
Let $\cF=\cF_S(G)$ with $S\in\Syl_p(G)$ and let $x\in S$ be fully $\cF$-centralised, so that $C_S(x)\in\Syl_p(C_G(x))$ \cite[Part I, Prop. 2.5]{AKO11} (see Lemma~\ref{lem:hom-count}(c)) and $x^\cF=x^G\cap S$. Then Theorem~\ref{thm:classsize} reads
\[
|x^G\cap S|_p=[S:C_S(x)]=\frac{|G|_p}{|C_G(x)|_p}=|x^G|_p .
\]
Conversely, this identity proves Theorem~\ref{thm:classsize} in the realisable case without any biset theory. The proof below was communicated to us by John Shareshian \cite{S26}; we record it since it is independent of everything else in this paper.
\end{rem}

\begin{prop}\label{p:shareshian}
Let $G$ be a finite group and $p$ be a prime. Let $S$ be a Sylow $p$-subgroup of $G$. Then for all $x \in S$, $$|x^G|_p = |x^G \cap S|_p.$$
\end{prop}

\begin{proof}
We double count the set $$\cX:=\{(y,Q) \mid Q \in \Syl_p(G) \mbox{ and } y \in x^G \cap Q\}.$$

Since $|x^G \cap Q|=|x^G \cap S|$ for all $Q \in \Syl_p(G)$, we have $|\cX|=|\Syl_p(G)| |x^G \cap S|$. As $|\Syl_p(G)| \equiv 1 \bmod p$, we have: $$|\cX|_p = |\Syl_p(G)|_p|x^G \cap S|_p=|x^G \cap S|_p.$$ On the other hand, $$|\cX|=|x^G|\cdot |\{Q \mid Q \in \Syl_p(G) \mbox{ and } x \in Q\}|.$$ Now $\langle x \rangle$ acts by conjugation on the set $$\cY:=\{Q \mid Q \in \Syl_p(G) \mbox{ and } x \notin Q\},$$ and moreover this action has no fixed points since if $x$ normalizes $Q \in \Syl_p(G)$ then $\langle x \rangle Q$ is a $p$-subgroup of $G$ strictly containing $Q$, a contradiction. Hence every orbit has size divisible by $p$ and so $|\cY|$ is divisible by $p$, whence $|\Syl_p(G) \backslash \cY|$ is not divisible by $p$. Therefore, $$|\cX|_p=|x^G|_p |\Syl_p(G) \backslash \cY|_p = |x^G|_p,$$ as needed.
\end{proof}

\begin{cor}\label{c:shareshiancor}
Theorem \ref{thm:classsize} holds when $\cF=\cF_S(G)$ is realised by a finite group $G$ with Sylow $p$-subgroup $S$.
\end{cor}

\begin{proof}
It suffices to show that for any fully $\mathcal{F}$-centralised element $x \in S$, we have \begin{equation}\label{e:mo}
|C_S(x)||\Hom_{\mathcal{F}}(\langle x \rangle,S)|_p = |S|.
\end{equation} Since $C_S(x) \in \Syl_p(C_G(x))$, using Proposition \ref{p:shareshian} we obtain $$|C_S(x)||\text{Hom}_{\mathcal{F}}(\langle x \rangle,S)|_p = |C_G(x)|_p |x^G \cap S|_p = |C_G(x)|_p \cdot |x^G|_p =|G|_p=|S|,$$ completing the proof.
\end{proof}

\section{Proof of Theorem \ref{t:main}}\label{s:proof}

The two ingredients, Proposition~\ref{prop:selfdual} and Theorem~\ref{thm:classsize}, are connected by the following identity, which holds for any fusion system.

\begin{lem}\label{lem:identity}
Let $\cF$ be any fusion system on $S$, $B$ a basis of $\Ch(S)^\cF$ and $X=X_B(\cF)$. Then
\[
\Gamma_B=X\Lambda\overline{X}^{\,T},\qquad \Lambda=\mathrm{diag}\Big(\tfrac{|s^\cF|}{|S|}\Big)_{s\in\Fz},
\]
and hence
\[
\det\Gamma_B=\det\big(X\overline X^{\,T}\big)\cdot\prod_{s\in\Fz}\frac{|s^{\cF}|}{|S|}.
\]
\end{lem}

\begin{proof}
Let $\psi,\phi\in B$. Since $\psi,\phi$ are constant on $\cF$-classes and every $\cF$-class is a disjoint union of $S$-classes,
\[
\langle\psi,\phi\rangle=\sum_{t^S\in\cl(S)}\frac{\psi(t)\overline{\phi(t)}}{|C_S(t)|}
=\sum_{s\in\Fz}\psi(s)\overline{\phi(s)}\sum_{t^S\subseteq s^\cF}\frac{1}{|C_S(t)|},
\]
and $\sum_{t^S\subseteq s^\cF}|C_S(t)|^{-1}=|S|^{-1}\sum_{t^S\subseteq s^\cF}|t^S|=|s^\cF|/|S|$. This is the asserted matrix identity, and the determinant formula follows.
\end{proof}

\begin{proof}[Proof of Theorem~\ref{t:main}]
Write $d=\det\big(X\overline X^{\,T}\big)=\Vol(\Ch(S)^\cF)^2$, a positive integer which is a power of $p$ by Remark~\ref{r:integral}, applied to a finite group $G\ge S$ with $\cF=\cF_S(G)$, which exists by Park's theorem \cite{P16}.

By Proposition~\ref{prop:selfdual}, $v_p(\det\Gamma_B)=0$. Applying $v_p$ to Lemma~\ref{lem:identity},
\[
0=v_p(d)+\sum_{s\in\Fz}\big(v_p(|s^\cF|)-v_p(|S|)\big),\qquad\text{i.e.}\qquad v_p(d)=\sum_{s\in\Fz}\big(v_p(|S|)-v_p(|s^\cF|)\big).
\]
Each $s\in\Fz$ is fully $\cF$-centralised, so Theorem~\ref{thm:classsize} gives $v_p(|s^\cF|)=v_p([S:C_S(s)])$ and therefore
\[
v_p(|S|)-v_p(|s^\cF|)=v_p(|C_S(s)|).
\]
Hence $v_p(d)=\sum_{s\in\Fz}v_p(|C_S(s)|)$, i.e.\ $d_p=\prod_{s\in\Fz}|C_S(s)|$. As $d$ is a power of $p$ we have $d=d_p$, which is the theorem.
\end{proof}

\begin{cor}\label{cor:equiv}
For a saturated fusion system $\cF$ on $S$, Theorem~\ref{t:main} is equivalent to \cite[Conj.~5.5]{LS25}, i.e.\ to the assertion that $|C_S(s)|\cdot|s^\cF|_p=|S|$ for all $s\in\Fz$.
\end{cor}

\begin{proof}
One direction is the proof above. Conversely, assume Theorem~\ref{t:main}. By Lemma~\ref{lem:identity} and Proposition~\ref{prop:selfdual},
\[
\prod_{s\in\Fz}\frac{|S|}{|C_S(s)|\cdot|s^\cF|_p}=\frac{d_p}{\prod_{s}|C_S(s)|}=1 .
\]
By Lemma~\ref{lem:lower} each factor is at most $1$, so each factor equals $1$.
\end{proof}

\begin{cor}\label{cor:msform}
Let $\cF$ be a saturated fusion system on $S$ and $B$ a basis of $\Ch(S)^\cF$. Then
\[
\det\Gamma_B=\prod_{s\in\Fz}m_s,
\]
with $m_s=|s^\cF|/|s^S|$ as in Lemma~\ref{lem:split}. Conversely, for an arbitrary fusion system $\cF$ this identity is equivalent to the conclusion of Theorem~\ref{t:main} for $\cF$.
\end{cor}

\begin{proof}
By Lemma~\ref{lem:split}, $|s^\cF|/|S|=|s^S|m_s/|S|=m_s/|C_S(s)|$, so Lemma~\ref{lem:identity} reads $\det\Gamma_B=d\cdot\prod_s m_s/\prod_s|C_S(s)|$. Now apply Theorem~\ref{t:main}.
\end{proof}

\section{An example}\label{s:example}

We illustrate each step of the proof on the smallest exotic example for which the inductive argument of \cite{LS25} does not apply. As in \cite{PS21} we write $\cF(p^n,i,j)$ for the $j$th fusion system on \texttt{SmallGroup}$(p^n,i)$, and we take $p=3$ and $\cF=\cF(3^4,9,2)$, an exotic simple fusion system on the maximal class group $S=\texttt{SmallGroup}(3^4,9)$ of order $81$. Here $\cN:=N_\cF(S)=\cF_S(N)$ is realised by a group $N$ of order $162$, the essential subgroups of $\cF$ are two elementary abelian subgroups $V_1,V_2$ of order $9$ with $\Aut_\cF(V_i)\cong\SL_2(3)$, and exactly three $\cN$-classes of elements of $S$, represented by $\bz$, $\bu$ and $\bu'$ say, are fused in $\cF$. The restrictions of irreducible characters of $N$ to $S$ are listed in Table \ref{t:nchartabspecal}, where $\omega$ is a primitive cube root of unity and
\[
\alpha=2\sqrt{3}\cos(\pi/18), \qquad \beta=2\sqrt{3}\cos(13\pi/18), \qquad \gamma=2\sqrt{3}\cos(25\pi/18)
\]
are the three distinct roots of $x^3-9x-9$. Thus $\bk(\cN)=10$ and $\bk(\cF)=8$.

\medskip\noindent\emph{Classes and the numerical ingredient.} The $\cF$-classes, with fully $\cF$-centralised representatives in the order $g_1,g_2,g_3,g_4,g_8,g_9,g_{10},g_5$ of Table \ref{t:nchartabspecal}, have the following data, where $m_s=|s^\cF|/|s^S|$ is the integer of Lemma \ref{lem:split}:
\[
\begin{array}{c|cccccccc}
s & g_1 & g_2=\bz & g_3 & g_4 & g_8 & g_9 & g_{10} & g_5 \\ \hline
|s^\cF| & 1 & 38 & 3 & 3 & 6 & 6 & 6 & 18\\
|s^S| & 1 & 1 & 3 & 3 & 3 & 3 & 3 & 9\\
|C_S(s)| & 81 & 81 & 27 & 27 & 27 & 27 & 27 & 9\\
m_s & 1 & 38 & 1 & 1 & 2 & 2 & 2 & 2
\end{array}
\]
The only interesting column is that of $\bz$: its $\cF$-class consists of the two non-trivial central elements of $S$, each forming an $S$-class by itself, together with the four $S$-classes of size $9$ contained in $\bu^\cF\cup\bu'^\cF$, so that
\[
m_\bz=1+1+4\cdot\frac{|C_S(\bz)|}{|C_S(\bu)|}=1+1+4\cdot 9=38 .
\]
Since $38\equiv 2 \pmod 3$ we have $|\bz^\cF|_3=1=|\bz^S|$, in accordance with Theorem \ref{thm:classsize}; note that $|C_S(\bu)|=9<81=|C_S(\bz)|$, so $\bu$ and $\bu'$ are not fully $\cF$-centralised, consistent with $\pi(1_{\bu^S})=\frac{9}{38}1_{\bz^\cF}$ having positive valuation at $3$ in Proposition \ref{prop:average}(c). Finally $\prod_{s\in\Fz}|C_S(s)|=81^2\cdot 27^5\cdot 9=3^{25}$ and $\prod_{s\in\Fz}m_s=38\cdot 2^4=608$.

\medskip\noindent\emph{The lattice and the self-duality ingredient.} Since $\bz$, $\bu$ and $\bu'$ are fused in $\cF$, any indecomposable $\cF$-stable character other than $\chi_1$ must be of the form $\rho+\psi$ with $\rho \in \{\chi_2, \chi_3+\chi_4,\chi_5,\chi_6,\chi_7\}$ and $\psi \in \{\chi_8,\chi_9,\chi_{10}\}$; indeed, a character $\sum_i a_i\chi_i$ is $\cF$-stable if and only if $a_3=a_4$ and $a_2+a_3+a_5+a_6+a_7=a_8+a_9+a_{10}$. Thus, by Lemma \ref{l:ind}, the set
\[
B=\{\chi_1,\ \chi_5+\chi_8,\ \chi_2+\chi_8,\ \chi_6+\chi_8,\ \chi_7+\chi_8,\ \chi_7+\chi_9,\ \chi_7+\chi_{10},\ \chi_3+\chi_4+\chi_8\}
\]
is a basis for $\Ch(S)^\cF$. As the $\chi_i$ are sums of distinct irreducible characters of $S$, no two of them having a common constituent, the inner product $\langle\psi,\phi\rangle$ counts common irreducible constituents, and in the above order
\[
\Gamma_B=\begin{pmatrix}
1&0&0&0&0&0&0&0\\
0&4&2&2&2&0&0&2\\
0&2&4&2&2&0&0&2\\
0&2&2&3&2&0&0&2\\
0&2&2&2&3&1&1&2\\
0&0&0&0&1&3&1&0\\
0&0&0&0&1&1&3&0\\
0&2&2&2&2&0&0&6
\end{pmatrix},\qquad \det\Gamma_B=608=2^5\cdot 19 .
\]
This is prime to $3$, as Proposition \ref{prop:selfdual} asserts, and indeed $608=\prod_{s\in\Fz}m_s$ in accordance with Corollary \ref{cor:msform}. Equivalently, $608\,\Gamma_B^{-1}$ is an integral matrix all of whose entries are divisible by $2^4$, so that $\Gamma_B^{-1}\in M_8(\ZZ_{(3)})$ and $\Ch(S)^\cF\otimes\ZZ_{(3)}$ is self-dual, by Lemma~\ref{l:duallattice}(b).

\medskip\noindent\emph{Conclusion.} Reading the values of the characters in $B$ off Table \ref{t:nchartabspecal} gives
\[
X_B(\cF)=\begin{pmatrix}
1&1&1&1&1&1&1&1\\
8&-1&2&2&2+\alpha&2+\beta&2+\gamma&-1\\
8&-1&2&2&-1+\alpha&-1+\beta&-1+\gamma&2\\
9&0&3\omega&3\overline{\omega}&\alpha&\beta&\gamma&0\\
9&0&3\overline{\omega}&3\omega&\alpha&\beta&\gamma&0\\
9&0&3\overline{\omega}&3\omega&\beta&\gamma&\alpha&0\\
9&0&3\overline{\omega}&3\omega&\gamma&\alpha&\beta&0\\
10&1&4&4&-2+\alpha&-2+\beta&-2+\gamma&-2
\end{pmatrix},
\]
and Lemma \ref{lem:identity} reads
\[
608=\det\Gamma_B=\det\big(X_B\overline{X_B}^{\,T}\big)\cdot\frac{1\cdot 38\cdot 3\cdot 3\cdot 6\cdot 6\cdot 6\cdot 18}{81^8},
\]
whence $\det\big(X_B\overline{X_B}^{\,T}\big)=3^{25}=\prod_{s\in\Fz}|C_S(s)|$, as Theorem \ref{t:main} predicts.

\begin{table}
\tiny{

\renewcommand{\arraystretch}{1.6}
\centering
\caption{$S$-restricted irreducible characters of $N$ when $\cF=\cF(3^4,9,2)$}
\label{t:nchartabspecal}
\begin{tabular}{|c|c|c|c|c|c|c|c|c|c|c|}
\hline
$\chi$ & $g_1$ & $g_2=\bz$ & $g_3$ & $g_4$ & $g_5$ & $g_6=\bu$ & $g_7=\bu'$ & $g_8$ & $g_9$ & $g_{10}$ \\ \hline
$\chi_1$ & $1$ & $1$ & $1$ & $1$ & $1$ & $1$ & $1$ & $1$ & $1$ & $1$ \\
$\chi_2$ & $2$ & $2$ & $2$ & $2$ & $2$ & $-1$ & $-1$ & $-1$ & $-1$ & $-1$ \\

$\chi_3$ & $2$ & $2$ & $2$ & $2$ & $-1$ & $2$ & $-1$ & $-1$ & $-1$ & $-1$ \\

$\chi_4$ & $2$ & $2$ & $2$ & $2$ & $-1$ & $-1$ & $2$ & $-1$ & $-1$ & $-1$ \\

$\chi_5$ & $2$ & $2$ & $2$ & $2$ & $-1$ & $-1$ & $-1$ & $2$ & $2$ & $2$ \\

$\chi_6$ & $3$ & $3$ & $3\omega$ & $3\overline{\omega}$ & $0$ & $0$ & $0$ & $0$ & $0$ & $0$ \\

$\chi_7$ & $3$ & $3$ & $3\overline{\omega}$ & $3\omega$ & $0$ & $0$ & $0$ & $0$ & $0$ & $0$ \\

$\chi_8$ & $6$ & $-3$ & $0$ & $0$ & $0$ & $0$ & $0$ & $\alpha$ & $\beta$ & $\gamma$ \\

$\chi_9$ & $6$ & $-3$ & $0$ & $0$ & $0$ & $0$ & $0$ & $\beta$ & $\gamma$ & $\alpha$ \\

$\chi_{10}$ & $6$ & $-3$ & $0$ & $0$ & $0$ & $0$ & $0$ & $\gamma$ & $\alpha$ & $\beta$ \\ \hline
\end{tabular}
}
\end{table}

\begin{rem}
The same bookkeeping can be carried out for the remaining simple fusion systems on $p$-groups of order $p^4$, which are classified in \cite{M18} and listed in \cite[Tables 1--4]{PS21}; explicit bases of $\Ch(S)^\cF$ for the systems $\cG$ and $\cF_{(1)}$ on a Sylow $p$-subgroup of $\PSp_4(p)$, and tables of the relevant character values, are given in \cite[Appendix]{LS25}. By Corollary~\ref{cor:msform} and Theorem~\ref{thm:classsize}, in every case $\det\Gamma_B=\prod_{s\in\Fz}m_s$, and this is prime to $p$.
\end{rem}

\section{Concluding remarks}\label{s:remarks}

\begin{rem}
Theorem \ref{thm:classsize} is a statement about fusion systems in which no characters appear: for a saturated fusion system $\cF$ on $S$ and a fully $\cF$-centralised $x\in S$, the $p$-part of $|x^\cF|$ is $[S:C_S(x)]$. It would be desirable to have a proof which does not pass through $\omega_\cF$, for instance a direct argument from the saturation axioms in the spirit of Proposition \ref{p:shareshian}. Such a proof would make the derivation of Theorem \ref{t:main} depend on the characteristic idempotent only through Proposition \ref{prop:selfdual}.
\end{rem}

\begin{rem}
Although Proposition \ref{prop:selfdual} removes the need for an orthogonalising matrix, \cite[Qu. 5.2]{LS25} remains of interest: for which saturated fusion systems $\cF$ does there exist $D \in M_{m \times n}(\ZZ)$ with $(\det(\overline{D^T}D),p)=1$ such that $\overline{(DX)^T}DX$ is diagonal with entries $|C_S(s)|$ up to $p'$-factors? By Section \ref{s:groupcase} such a $D$ exists whenever $\cF=\cF_S(G)$ is realisable, and one may then take $D=D_B(G)$, or $D=D_B(B_0)$ for $B_0$ the principal $p$-block of $G$. If $\cF$ is the fusion system of the principal block $B_0$ of a spets $(\GG,q)$, the article \cite{KMS4} postulates the existence of a set $\Irr(B_0)$ of mutually orthogonal characters which specialise to virtual characters of $S$; the analogous matrix $D_B(B_0)$ could then be used to answer this question in the affirmative for some families of exotic fusion systems, such as the Benson--Solomon fusion systems (see \cite[Section 5]{KMS4} for partial calculations in this direction).
\end{rem}

\end{document}